\documentclass[a4paper, 11pt, final]{amsart}
\usepackage{multicol}
\usepackage{multirow}
\usepackage[pdftex]{graphicx}
\usepackage{amsmath, amsfonts, amssymb, amsthm, float}
\usepackage{mathtools}
\usepackage{mathrsfs}
\usepackage{showkeys}
\usepackage{enumitem}
\usepackage[a4paper, bottom=2cm]{geometry}
\usepackage{rotating}
\usepackage{array}
\usepackage{lmodern}
\usepackage[normalem]{ulem}
\usepackage{xcolor}
\usepackage{url}
\usepackage{tikz-cd}
\usepackage{tabularx}

\usepackage{etoolbox}

\usepackage[colorlinks=true,linkcolor=black,citecolor=blue!80!black]{hyperref}
\usepackage{xcolor}

\usepackage[capitalize]{cleveref}
\theoremstyle{plain}
\theoremstyle{definition}

\newtheorem{definition}{Definition}[section]

\newtheorem{corollary}[definition]{Corollary}

\newtheorem{lemma}[definition]{Lemma}

\newtheorem{theorem}[definition]{Theorem}
\newtheorem{proposition}[definition]{Proposition}

\newcounter{claim}[definition]

\theoremstyle{remark}

\theoremstyle{remark}

\theoremstyle{definition}

\renewcommand{\epsilon}{\varepsilon}
\renewcommand{\bar}{\overline}
\renewcommand{\hat}{\widehat}
\renewcommand{\leq}{\leqslant}
\renewcommand{\geq}{\geqslant}

\newcommand{\N}{\mathbb{N}}
\newcommand{\Z}{\mathbb{Z}}
\newcommand{\Q}{\mathbb{Q}}

\newcommand{\FF}{\mathbb{F}}
\newcommand{\SL}{\operatorname{SL}}

\newcommand{\syl}{\mathrm{Syl}}

\newcommand{\GL}{\mathrm{GL}}

\newcommand{\0}{\emptyset}

\newcommand{\PSL}{\mathrm{PSL}}

\newcommand{\Aut}{\mathrm{Aut}}

\def \H {\mathrm H}
\def \ov {\overline}

\renewcommand{\hat}{\widehat}

\numberwithin{equation}{section}
\makeatletter
\let\c@theorem\c@equation
\makeatother

\begin{document}

\author{Chris Parker}
\address{School of Mathematics, University of Birmingham, B15 2TT, United Kingdom}
\email{c.w.parker@bham.ac.uk}

\author{Martin van Beek}
\address{Department of Mathematics, University of Manchester, M13 9PL, United Kingdom}
\email{martin.vanbeek@manchester.ac.uk}


\title{Lifting Polynomial Representations of $\SL_2(p^r)$ from $\FF_p$ to $\Z/p^s\Z$}

\begin{abstract}
We describe all of the irreducible polynomial $\FF_p\SL_2(p^r)$ representations which lift to $(\Z/p^s\Z)\SL_2(p^r)$ representations for $s>1$, observing that they almost never do. We also show that two related indecomposable $\FF_p \SL_2(p^r)$ representations cannot be lifted to $\Z/p^s\Z$ representations for $s>1$.
\end{abstract}

\maketitle

 \section{Introduction}

The lifting problem (see \cref{liftdefn}) for modular representations of finite groups has been considered by different authors in a variety of settings over many years, and we mention only a small selection of this research  here. The theory has its origin in the work of Brauer. For a historical overview of the problem, we refer to Benson's text \cite[3.7]{benson}.

Throughout, we let $p$ be a prime, $\FF_{p^r}$ be the field of $p^r$ elements, $\mathbb{Q}_p$ the $p$-adics and $\Z_p$ the $p$-adic integers. Green \cite{Green} gave a sufficient homological condition for an $\FF_p$-representation to lift to a $\Z_p$-representation. In particular, this demonstrates that the inability of an $\FF_p$-representation to lift gives information about certain $\mathrm{Ext}$ groups of the module associated to the representation. Work of Tiep--Zalesski \cite[Theorem 4.5]{TiepZalesski} revealed conditions for $\FF_p$-representations of groups of Lie type in characteristic $p$ to lift to representations over $\Z/p^s\Z$. In this direction, our contribution is a finer analysis of when the homogeneous polynomial representations $V_i(p^r)$, $1 \le i \le p$, of $\SL_2(p^r)$ lift to $\Z/p^s\Z$ for $s\geq 2$, extending the work in \cite[Theorem 4.5]{TiepZalesski}.

\begin{theorem}\label{MT}
Suppose that $p$ is a prime, $r\in \N$, $G=\SL_2(p^r)$ and $V$ is an $\FF_pG$-module. If, for some $1\leq i\leq p$, we have that $V=V_i(p^r)$ or $V=\Lambda(p^r)$, then $V$ lifts to $\Z/p^2\Z$ if and only if $r=1$ and either
\begin{enumerate}
    \item $p=2$ and $V\in\{V_1(2), V_2(2)\}$; or
    \item $p$ is odd and $V\in\{V_{p-2}(p), V_{p-1}(p)\}$.
\end{enumerate}
Moreover, if $V$ lifts to $\Z/p^2\Z$, then $V$ lifts to $\Q_p$.
\end{theorem}

We recall that $V$ lifts to $\Q_p$ if there is a $\Q_pG$-module $U$ and a $G$-stable $\Z_p$-lattice $W\subseteq X$ such that $W/pW\cong V$ as $\FF_pG$-modules. We provide a description of the modules $V_i(p^r)$ and $\Lambda(p^r)$ in \cref{sec:structureV}. The proof of \cref{MT} is a rather intricate calculation.

A subgroup $H$ of $\GL_n(\Z/p^s\Z)$ is said to be \emph{irreducible} if $H$ preserves no proper non-trivial free $\Z/p^s\Z$-submodule of the natural $\Z/p^s\Z$-module for $\GL_n(\Z/p^s\Z)$ upon restriction. Equivalently, $H$ is an irreducible subgroup of $\GL_n(\Z/p^s\Z)$ if $\pi(H)$ is an irreducible subgroup of $\GL_n(\Z/p\Z)$, where $\pi$ is the natural projection $\pi :\GL_n(\Z/p^s\Z) \rightarrow \GL_n(\Z/p\Z)$. Our theorem, along with that of \cite[Theorem 4.5]{TiepZalesski} and \cite{KS}, addresses the broader question of which quasisimple groups are \emph{irreducible} subgroups of $\GL_n(\Z/p^s\Z)$ with $s \ge 1$. We believe this question to be of fundamental interest and whose answer may have broader applications.

\sloppy{Set $Q=O_p(\GL_n(\Z/p^2\Z))$ and assume that $W$ is the natural $\Z/p^2\Z$-module for $\GL_n(\Z/p^2\Z)$. If $G \le \GL_n(\Z/p^2\Z)$ and $U=W/pW$ is regarded as an $\FF_pG$-module, then $Q \cong U\otimes_{ \FF_p}U^*$ as an $\FF_pG$-module with action given by conjugation. With this in mind, our theorem shows that with $V$ and $G$ as described in \cref{MT}, the failure of $G$ to lift to $\Z/p^2\Z$ exhibits non-zero second degree cohomology:}

\begin{corollary}
With $V$ and $G$ as described in \cref{MT} we have $\mathrm H^2(G,V\otimes_{\FF_p} V^*)\ne 0$ unless $V\in\{V_{p-2}(p), V_{p-1}(p), V_2(2)\}$. Equivalently, $\mathrm{Ext}_{\FF_pG}^2(V,V)\ne 0$ unless $V\in\{V_{p-2}(p), V_{p-1}(p), V_2(2)\}$.
\end{corollary}

Our methodology for proving \cref{MT} is as follows: Recognizing $V$ as an $\FF_{p^r}\SL_2(p^r)$-module, we write $k:=\dim_{\FF_{p^r}}V$. We observe that if $V$ lifts to $\Z/p^s\Z$ for some $s>1$ then it lifts to $\Z/p^2\Z$, and so we assume that there is a subgroup $H$ of $\GL_{kr}(\Z/p^2\Z)$ isomorphic $\mathrm{(P)SL}_2(p^r)$ such that the image of $H$ in $\GL_{kr}(\FF_p)$ under the natural projection corresponds to the action of $\SL_2(p^r)$ on $V$. In particular, the Sylow $p$-subgroups of $H$ are elementary abelian and there is a toral element in $H$ which preserves a Sylow $p$-subgroup. For $a\in H$ of order $p$, we consider the elements of order $p$ in the coset $aQ$, where $Q=O_p(\GL_{kr}(\Z/p^2\Z))$. The presence of such an element matching the action of $a$ provides certain equations in $\mathrm{M}_{kr}(\Z/p^2\Z)$ which must be satisfied. Then the action of a toral element on the elements of order $p$ in $H$ provides yet more equations which must also be satisfied.

Fortunately, the equations produced by focusing on only a handful of $r\times r$-block matrices provide enough restriction to determine the candidates for $V$ in \cref{MT} which do not lift to $\Z/p^2\Z$. Because of this, we never need to determine complete $(kr\times kr)$-matrices and this is one of the reasons our approach to this problem is successful. Finally, the modules in \cref{MT} are either projective, and lift to $\Q_p$ by \cref{Projective}, or have been shown to lift to $\Q_p$ in \cite[Proposition 3.8]{p.index3}, completing the proof of the result.

The authors' own interest in this result stems from its application to classifying saturated fusion systems with a normal abelian essential subgroup, as part of a project with other collaborators. This future work aims to generalize work of Oliver and others (\cite{p.index1}, \cite{p.index2}, \cite{p.index3}), and the results in this paper are inspired by the results of Oliver and Ruiz in the final paper of that work.

\textit{Acknowledgements:} The second author is supported by the Heilbronn Institute for Mathematical Research. The authors wish to thank Valentina Grazian, Justin Lynd, Bob Oliver and Jason Semeraro for their input during the workshop ``Patterns in Exotic Fusion Systems" and the Heilbronn Institute for Mathematical Research for funding the workshop.

We thank the anonymous referee for their comments which greatly improved the exposition of this work.

\section{Preliminaries}

We are concerned with the question of when modules defined over $\FF_p$ \emph{lift} to $\Z/p^s \Z$ for various $s\in \N$.

\begin{definition}\label{liftdefn1}
Let $G$ be a finite group, $\theta: G \rightarrow \GL_n(p)$ a representation of $G$. Then $(G, \theta)$ is \emph{liftable} to $\Z/p^s\Z$ if there is $\widehat \theta : G \rightarrow \GL_n(\Z/p^s\Z)$  such that with $\pi_s:\GL_n(\Z/ p^s\Z)\rightarrow \GL_n(p)$ the natural projection, we have $\theta=\widehat \theta \circ\pi_s$.
\end{definition}

Identifying $H$ with the image of $\theta$, we can reformulate \cref{liftdefn1} as follows.

\begin{definition}\label{liftdefn}
Let $G$ be a finite group, $V$ a faithful $n$-dimensional $\FF_pG$-module and identify $G\le \GL_n(p)$. Then $(G, V)$ is \emph{liftable} to $\Z/p^s\Z$ if there is $H\le \GL_n(\Z/p^s\Z)$ with $H$ mapping isomorphically to $G$ under the natural projection $\pi_s$.
\end{definition}

A further formulation is that there is an $n$-dimensional $(\Z/ p^s\Z)G$-module $W$ which is free as a $\Z/p^2\Z$-module (so a direct product of $n$ copies of $C_{p^s}$ on which $G$ acts) such that $W/pW$ is isomorphic to $V$ as an $\FF_pG$-module. We will generally suppress the group $G$ and simply say that $V$ lifts.

\begin{lemma}\label{lem: ppower}
Let $G$ be a finite group, and $W$ a $(\Z/p^2\Z)G$-module with $W$ free as a $\Z/p^2\Z$-module. Then $pW\cong W/pW$ as $\FF_p G$-modules.
\end{lemma}
\begin{proof}
We have that $pW$ and $W/pW$ are elementary abelian $p$-groups. Set $\phi: W\to pW$ to be the $p$th-power map on $W$. Then $\phi$ commutes with the $G$-action on $W$, and so $\phi$ is a $G$-map with kernel $pW$. Hence, $W/pW\cong pW$ as $\FF_p G$-modules.
\end{proof}

We record a result of Green which reiterates our comments from the introduction and provides a characterization of when modules can lift.

\begin{theorem}\cite[Theorem 1]{Green}\label{Green}
Let $G$ be a finite group and $V$ an $\FF_p G$-module such that $\mathrm{Ext}^2_{\FF_pG}(V, V)=0$. Then $V$ lifts to a $\Q_p G$-module.
\end{theorem}

We remark that the condition $\mathrm{Ext}^2_{\FF_pG}(V, V)=0$ is equivalent to the requirement that every extension of $G$ by $V\otimes_{\FF_p} V^*$ splits as a semidirect product \cite[Proposition 3.1.8, Proposition 3.7.3]{benson}.

\begin{corollary}\label{Projective}
 Let $G$ be a finite group and $V$ a finitely generated, projective $\FF_p G$-module. Then $V$ lifts to a $\Q_p G$-module.
\end{corollary}
\begin{proof}
Since $V$ is projective and finitely generated, so too is $V\otimes_{\FF_p} V^*$. Hence, by \cite[Proposition 3.1.8]{benson} we see that $\mathrm{Ext}^2_{\FF_pG}(V, V)=\H^2(G, V\otimes_{\FF_p} V^*)=0$, as required.
\end{proof}

Another consequence of \cref{Green} is the following corollary which motivates our approach to proving the main theorem.

\begin{corollary}\label{CorGreen}
Let $G$ be a finite group, $S\in\syl_p(G)$, $H:=N_G(S)$ and $V$ an $\FF_pG$-module. Assume that $T\in\syl_p(G)$ satisfies $S\cap T=1$ whenever $S\ne T$. Then $V$ lifts to a $\Z_p G$-module if and only if $V|_H$ lifts to an $\Z_p H$-module.
\end{corollary}

Of course, if an $\FF_pG$-module $V$ lifts to $\Z/p^s\Z$ where $s>1$ then $V$ lifts to $\Z/p^2\Z$ intermediately.

We  use the remainder of this section to record some useful arithmetical and combinatorial results which will be required to prove the main theorem.

Recall that a \emph{Pascal matrix} $a$ is a lower triangular matrix which encodes the binomial coefficients. That is, whenever $i\geq j$ the $(i,j)$th entry $a_{i,j}$ of $a$ is $\binom{i-1}{j-1 }$ and whenever $i<j$ then $a_{i,j} =0$.

\begin{lemma}\label{Pascal}
Let $a_0=a-I$ where $a$ is an $n\times n$ Pascal matrix. Then, for $\ell \geq 1$, the following hold:
\begin{enumerate}
\item \label{pasc1} $(a_0^{\ell})_{\ell+1,1}= \ell!$ and $(a_0^{\ell})_{j,1}= 0$ for all $j\leq \ell$.
\item \label{pasc2} $(a_0^{\ell})_{\ell+2,2}=(\ell+1)!$ and $(a_0^{\ell-1})_{j,2}= 0$ for all $j\leq \ell+1$.
\item \label{pasc3} $(a_0^{\ell})_{\ell+2,1}=\binom{\ell+1}{\ell-1}\ell!$ and $(a_0^{\ell-1})_{j,2}= 0$ for all $j<\ell+1$.
\end{enumerate}
\end{lemma}
\begin{proof}
The result in (\ref{pasc1}) is true for $\ell=1$ by definition.  It is also easy to see that the entries $ (a_0^{\ell})_{j,1}=0 $ for $j\leq \ell$. In particular, if $n\leq l$ then $(a_0)^l=0_n$. To complete the proof, we assume that $l<n$. Write $a_0^{\ell}= a_0(a_0^{\ell-1})$ and apply induction to see that
\[(a_0^{\ell})_{\ell+1,1}= \sum_{i=1}^n (a_0)_{\ell+1, i}(a_0^{\ell-1})_{i,1} =(a_0)_{\ell+1,\ell}(a_0^{\ell-1})_{\ell,1}= \binom{\ell}{\ell-1}(\ell-1)!= \ell!\]
as claimed in (\ref{pasc1}).

The inductive steps for parts (\ref{pasc2}) and (\ref{pasc3}) are
\[(a_0^{\ell})_{\ell+2,2}= \sum_{i=1}^n (a_0)_{\ell+2, i}(a_0^{\ell-1})_{i,2} =(a_0)_{\ell+2,\ell+1}(a_0^{\ell-1})_{\ell+1,2}= \binom{\ell+1}{\ell}\ell!= (\ell+1)!\]
and
\begin{eqnarray*}(a_0^{\ell})_{\ell+2,1}&=& \sum_{i=1}^n (a_0)_{\ell+2, i}(a_0^{\ell-1})_{i,1}=
(a_0)_{\ell+2,\ell}(a_0^{\ell-1})_{\ell,1}+(a_0)_{\ell+2,\ell+1}(a_0^{\ell-1})_{\ell+1,1}
\\& =&\binom{\ell+1}{\ell-1}(\ell-1)!+\binom{\ell+1}{\ell}\binom{\ell}{\ell-2}(l-1)!=\binom{\ell+1}{\ell-1}\ell!.\end{eqnarray*}
\end{proof}

\section{The structure of $V_n(p^r)$ and $\Lambda(p^r)$}\label{sec:structureV}

Let $p$ be a prime and $r\in \N$. Set $A = \FF_{p^r}[x,y]$ to be the ring of polynomials in two commuting indeterminates $x,y$ with coefficients in $\FF_{p^r}$. Make $A$ into an $\FF_{p^r} \SL_2(\FF_{p^r})$-module as follows: for $g =\left(\begin{smallmatrix} a & b \\ c & d\end{smallmatrix}\right) \in \SL_2(\FF_{p^r} )$ define $x\cdot g=(a x + b y)$ and $y\cdot g= (c x + d y)$. Observe that the action of $\SL_2(\FF_{p^r})$ on $A$ maps polynomials of degree $n$ to polynomials of degree $n$.

For $n\geq 1$ let $V_n(p^r) \le  A$ be the set  of homogeneous polynomials of degree $n$. Since the action of $\SL_2(\FF_{p^r})$ on $A$ preserves degree, $V_n(p^r)$ is an $\FF_{p^r} \SL_2(\FF_{p^r})$-submodule of $A$ of dimension $n+1$. Provided $0 \leq n \leq p-1$, the modules $V_n(p^r)$ are irreducible self-dual $\FF_{p^r} \SL_2(\FF_{p^r})$-modules \cite[p588]{BrauerNesbitt}. They are important because of the Steinberg Tensor Product Theorem \cite[Corollary 2.8.6]{GLS3} which asserts that each irreducible $\FF_{p^r} \SL_2(\FF_{p^r})$-module may be written as
\[ U_0\otimes U_1^\sigma\otimes \dots\otimes U_{a-1}^{\sigma^{a-1}}\] where $U_i\cong V_j(p^r)$ for some $j\in\{0, \dots, p-1\}$ and $\sigma$ is the Frobenius automorphism which maps $\lambda \mapsto \lambda^p$ for all $\lambda \in \FF_{p^r}$. As such, the Frobenius twists of these modules provide a complete list of the so-called \emph{basic} $\FF_{p^r} \SL_2(\FF_{p^r})$-modules (see \cite[Example 2.8.10a]{GLS3}). We will often view $\FF_{p^r}\SL_2(\FF_{p^r})$-modules as $\FF_p\SL_2(\FF_{p^r})$-modules. That is, we recognize that an $\FF_{p^r}\SL_2(\FF_{p^r})$-module is an $\FF_p$-vector space with an action of $\SL_2(\FF_{p^r})$. The collection of basic $\FF_{p^r} \SL_2(\FF_{p^r})$-modules, viewed instead as $\FF_p\SL_2(\FF_{p^r})$-modules, are what we call \emph{irreducible polynomial} modules.

We observe that when $p^r>2$, $V_p(p^r)$ is indecomposable but not irreducible. Indeed, $V_p(p^r)$ has a $2$-dimensional submodule generated by $x^p$ and $y^p$ (see \cite[Lemma 2.1]{GPSvB2025}). Denoting this submodule by $U$, we have that the restriction of $U$ to $\FF_p\SL_2(\FF_{p^r})$ is isomorphic to the restriction of $V_1(p^r)$ to $\FF_p\SL_2(\FF_{p^r})$; while the restriction of $V_p(p^r)/U$ to $\FF_p\SL_2(\FF_{p^r})$ is isomorphic to the restriction of $V_{p-2}(p^r)$ to $\FF_p\SL_2(\FF_{p^r})$.

We will also be interested in the dual module of $V_p(p^r)$, which we denote by $\Lambda(p^r)$, since this module plays a role in a current project of the authors. Indeed, the Sylow $p$-subgroup of the semidirect product of $\SL_2(\FF_{p^r})$ by $\Lambda(p^r)$ supports exotic fusion systems \cite{GPSvB2025}. By a dualization argument, this module contains a $(p-1)$-dimensional $\FF_{p^r}\SL_2(\FF_{p^r})$-module whose restriction to $\FF_p\SL_2(\FF_{p^r})$ is isomorphic to the restriction of $V_{p-2}(p^r)$ to $\FF_p\SL_2(\FF_{p^r})$; while the restriction to $\FF_p\SL_2(\FF_{p^r})$ of the quotient of $\Lambda(p^r)$ by this submodule is isomorphic to the restriction of $V_1(p^r)$ to $\FF_p\SL_2(\FF_{p^r})$. If $p^r=2$, we observe that $\Lambda(2)\cong V_2(2)\cong V_1(2)\oplus V_0(2)$, where $V_0(2)$ is the trivial module for $\FF_2\SL_2(2)$.

Throughout the remainder of this article, we will identify $\SL_2(\FF_{p^r})$ with $\SL_2(p^r)=:G$ and view $V_n(p^r)$ and $\Lambda(p^r)$ as $\FF_pG$-modules.

\section{Lifting $V_n(p^r)$ and $\Lambda(p^r)$ to representations over $\Z/p^s Z$}

In this section, we demonstrate that for $1\leq n\leq p$ and $s>1$ the modules $V_n(p^r)$ lift to $\Z/p^s\Z(\SL_2(p^r))$-representations if and only if $p^r=2$; or $r=1$, $p$ is odd and $n\in\{p-1,p-2\}$. We begin with a classical result of Higman which deals with one case when $p=2$.

\begin{theorem}\label{p=2case}
For any $s>1$, we have that $ V_1(2^r)$ lifts to $\Z/2^s\Z$ if and only if $r=1$. In particular, $V_1(2)$ lifts to $\Q_2$.
\end{theorem}
\begin{proof}
The first statement is \cite[Theorem 8.2]{Higman}. For the second statement, we observe that $V_1(2)$ is projective so that $V_{1}(2)$ extends to a $\Q_2\SL_2(2)$-representation by \cref{Projective}.
\end{proof}

Set $G:=\SL_2(p^r)$ and assume that $W$ is a free $\Z/p^2\Z$-module of rank $r(n+1)$ equipped with a $G$-action such that, as an $\FF_pG$-module, $W/pW \cong V_n(p^r)$ where $1\leq n\leq p$. Let $\mathrm{M}_{(n+1)r}(\Z/p^2\Z)$ denote the algebra of $(n+1)r\times (n+1)r$ matrices with entries from $\Z/p^2\Z$ and identify $\Aut(W)$ with $\GL_{(n+1)r}(\Z/p^2\Z)$: the subset of $\mathrm{M}_{(n+1)r}(\Z/p^2\Z)$ whose elements have invertible determinant. Let $I_r$ be the identity $r\times r$ matrix, $I_{(n+1)r}$ the identity $(n+1)r\times (n+1)r$ matrix and $0_r$ the zero $r\times r$ matrix. Let $H$ be the image of $G$ in $\Aut(W)$ so that $H\cong G=\SL_2(p^r)$ if $n$ is odd, while $H\cong G/Z(G)=\PSL_2(p^r)$ if $n$ is even.

Let $\alpha:=\begin{pmatrix} 1&0\\1&1\end{pmatrix}\in G$ and $\gamma:=\begin{pmatrix} \lambda&0\\0&\lambda^{-1} \end{pmatrix}\in G$ where $\lambda\in \FF_{p^r}$ has order $p^r-1$. Then $\alpha^p=[\alpha,\alpha^\gamma]=1$. For $0\leq k\leq n$, upon viewing the image of $\alpha$ in $H$ as a $(n+1)\times (n+1)$ block matrix with blocks of size $r$, the action of $\alpha$ on $V_n(p^r)$ may be represented by the matrix with $(k+1)$-th row
\[\left(I_r, \binom{k}{1}I_r, \binom{k}{2}I_r, \binom{k}{3}I_r, \dots ,\binom{k}{k-1}I_r, I_r, \underbrace {0_r,\dots, 0_r}_{n-k\,\text{terms}}\right).\]
That is, $\alpha$ is the $r\times r$-block version of a lower triangular $(n+1)\times (n+1)$ Pascal matrix. Turning to $\gamma$, it acts as an $r\times r$-block diagonal matrix:
\[\mathrm{diag}(\bar{C}^{n},\bar{C}^{n-2}, \dots, \bar{C}^{-(n-2)}, \bar{C}^{-n}) \in \GL_{(n+1)r}(p)\]
where $\bar{C}\in \GL_r(p)$ has order $p^r-1$.  Notice that $\langle \alpha,\gamma\rangle$ is the normalizer of a Sylow $p$-subgroup of $G$  and so studying $\langle \alpha,\gamma\rangle$ is significant because of \cref{CorGreen}.

Now $C_{\Aut(W)}(W/pW)$ - the kernel of the natural map $\Aut(W)\to \Aut(W/pW)$ - is the subgroup
\[Q= \{I_{(n+1)r}+q\mid q\in p\mathrm{M}_{(n+1)r}(\Z/p^2\Z)\}\]
and $\Aut(W)/Q \cong \GL_{(n+1)r}(p)$. We choose $a, c \in \Aut(W)$ so that $aQ$ and $cQ$ correspond to the matrices defined for $\alpha$ and $\gamma$ respectively. Furthermore, we may choose $a$ and provisionally choose $c$ so that they have the same matrix entries as the images in $H$ of $\alpha$ and $\gamma$ but read in $\Z/p^2\Z$. In this way, $a$ has order dividing $p^2$ and $c$ provisionally has order dividing $p(p^r-1)$. As $H \cap cQ $ is non-empty, by Hall's Theorem \cite[Theorem 6.4.1]{GOR} and after possibly replacing $c$ by $c^p$ and $H$ by a conjugate of $H$ by some element of $Q$, we may assume that $c \in H$. Hence $c$ has order $p^r-1$ when $p=2$ and order $\frac{p^{r}-1}{(n,2)}$ when $p$ is odd. We may represent $c$ by an $r\times r$-block matrix of the form 
\[\mathrm{diag}(C^{n},C^{n-2}, \dots, C^{-(n-2)}, C^{-n}) \in \GL_{(n+1)r}(\Z/p^2\Z), \tag{4.1.1}\]
where $C\in \GL_{r}(\Z/p^2\Z)$ has order $p^r-1$, and $\bar{C}$ is the image of $C$ under the natural map from $\GL_{(n+1)r}(\Z/p^2\Z)$ to $\GL_{(n+1)r}(p)$. Furthermore, we know that $H \cap aQ$ is non-empty and so $aQ$ must contain elements of order $p$.

We first investigate the elements of order $p$ in $aQ$. Such elements have the form $a+q$ for some $q\in pM_{(n+1)r}(\Z/p^2\Z)$. We write $a_0=a-I_{(n+1)r}$. Then $a_0$ is a strictly lower triangular matrix so that $a_0^{n+1}=0_{(n+1)r}$ and, as every entry of $q$ is a multiple of $p$, $q^2=pq=0_{(n+1)r}$. We calculate in $\GL_n(\Z/p^2\Z)$ that
\begin{align*}\label{calc 1}
(a+q)^p & = (I_{(n +1)r}+(a_0+q))^p\\
 ~ & = I_{(n+1)r}+p(a_0+q)+\binom{p}{2}(a_0+q)^2+\dots +p(a_0+q)^{p-1}+(a_0+q)^p\nonumber \tag{4.1.2}\\
 ~ & = I_{(n +1)r}+pa_0+\binom{p}{2}a_0^2+\dots+ \binom{p}{p-2}a_0^{p-2}+pa_0^{p-1}\nonumber\\
 ~ & ~ \quad +a_0^p+ a_0^{p-1}q + a_0^{p-2}qa_0+\dots + qa_0^{p-1}\nonumber
\end{align*}
because any appearance of at least two $q$s in a product makes the matrix have entries a multiple of $p^2$ and so contributes the zero matrix $0_{(n+1)r}$. For the same reason, any appearance of $q\binom{p}{i}$ (for $1\leq i<p$) yields entries which are multiples of $p^2$, since $\binom{p}{i}$ is divisible by $p$, and results in a zero matrix.

Now regard $(a+q)^p$ as a $(n+1)\times (n+1)$ block matrix with $r \times r$ entries in each block and, for $1\leq i, j \leq n+1$, denote by $R_{i,j}$ the $r\times r$ block in position $(i,j)$. Then the matrices $I_{(n+1)r}$, $a_0^j$, and $a_0^jqa_0^{p-1-j}$ with $j \geq 2$ contribute $0_r$ to $R_{2,1}$. Thus, $R_{2,1}$ is the summation of the $(2,1)$ blocks of $pa_0$, $a_0qa_0^{p-2}$ and $qa_0^{p-1}$. Since we require $(a+q)^p=I_{(n+1)r}$, we must have $R_{2,1}= 0_r$.

We have that $H \le \Aut(W)$, $H$ has abelian Sylow $p$-subgroups and $c$ is contained in the normalizer of a Sylow $p$-subgroup of $H$. Thus, for some choice of $a+q$ of order $p$, we have $(a+q)(a+q)^c=(a+q)^c(a+q)$. We calculate for arbitrary $b=a+q$, that
\begin{align*}\label{calc 2}
bb^c-b^cb&= (I_{(n+1)r}+a_0+q)(I_{(n+1)r}+a_0+q)^c-
(I_{(n+1)r}+a_0+q)^c(I_{(n+1)r}+a_0+q)\\
&= (a_0+q)(a_0+q)^c-(a_0+q)^c(a_0+q)\tag{4.1.3} \\&= (a_0a_0^c+qa_0^c+a_0q^c)-(a_0^ca_0+a_0^cq +q^ca_0)
\end{align*}
as $qq^c=q^cq=0_r$.

Let $Q_{i,j}$ denote the $r\times r$ block in position $(i,j)$ in $q$. Since $Q_{i,j}\in p\mathrm{M}_{r\times r}(\Z/p^2\Z)$, we know that $pQ_{i,j}=0$. We shall denote the $r\times r$-block in position $(i,j)$ of $a_0^k$ by $(a_{0}^k)_{(i,j)}$.

We are now ready to prove the main results of this paper. The following proposition is also a consequence of \cite[Lemma 4.3]{TiepZalesski}, as using their terminology the $p$-elements from $G$ are not \emph{coreless} in this case. In our situation, the result also follows from work of Th\'evenaz \cite{Thevenaz} generalizing the results in \cite{Hannula} and \cite{DDVP}. We note that the result for $V_1(p)$ was already known to Serre in 1968 \cite[IV Lemma 3]{Serre}.

\begin{proposition}\label{the generic case}
Assume that $p>3$ is an odd prime and $G=\SL_2(p^r)$ where $r\geq 1$. Let $1\leq n\leq p-3$. Then $V_n(p^r)$ does not lift to $\Z/p^2\Z$.
\end{proposition}
\begin{proof}
Suppose that $1\leq n\leq p-3$. We proceed by contradiction, noting that if a lift exists then $aQ$ contains an element of order $p$, namely $a+q$ for some $q\in Q$. As in the discussion above let $R_{i,j}$ be the $r\times r$ block of $(a+q)^p$. Then, as $a+q$ has order $p$,
\[ R_{i,j}=
\begin{cases}
I_r & i=j\\
0_r & \text{otherwise.}
\end{cases}\]

Since $a_0^{n+1}=0$ and $n \leq p-3$ we see that the terms $a_0qa_0^{p-2}$ and $qa_0^{p-1}$ from (\ref{calc 1}) are equal to $0_{r(n+1)}$, and so their contribution to $R_{i,j}$ is zero for all $i,j$. In particular, as $R_{2,1}$ is the sum of the $(2,1)$-blocks from $pa_0$, $a_0qa_0^{p-2}$ and $qa_0^{p-1}$, we have that \[0_r=R_{2,1}=p{a_0}_{(2,1)}=pI_r\ne 0_r,\] which is a contradiction. Hence no lift of $V_n(p^r)$ to $\Z/p^2\Z$ exists.
\end{proof}

\begin{proposition}\label{the harder calculation}
Assume that $p$ is an odd prime and $G=\SL_2(p^r)$. Then $V_{p-2}(p^r)$ lifts to $\Z/p^2\Z$ if and only if $r=1$. Moreover, $V_{p-2}(p)$ lifts to $\Q_p$.
\end{proposition}
\begin{proof}
Here we work in $\mathrm{M}_{(p-1)r}(\Z/p^2\Z)$ and set $n=p-2$. We continue to assume that $(a+q)^p=1$ and use (\ref{calc 1}). Then ${a_0}^{p-1}={a_0}^{n+1}$ is equal to $0_{r(p-1)}$.  Since the $R_{2,1}$ contribution comes from the corresponding submatrices of $pa_0$, $a_0qa_0^{p-2}$ and $qa_0^{p-1}$, and as the only non-zero $r\times r$ block of $a_0^{p-2}$ is in position $(p-1,1)$, we have

\[R_{2,1}=p{a_0}_{(2,1)} + {a_0}_{(2,1)}.Q_{1,p-1}.({a_0}^{p-2})_{(p-1, 1)}.\]

By definition we see that ${a_0}_{(2,1)} = I_r$. Furthermore, \cref{Pascal} yields that ${(a_0^{p-2})}_{(p-1, 1)}=(p-2)! I_r$. Recall that $(p-2)!\equiv 1 \pmod p$ so that
\[{a_0}_{(2,1)}.Q_{1,p-1}.(a_0^{p-2})_{(p-1, 1)}=Q_{1,p-1}.(tp+1)=Q_{1,p-1}\]
since $pQ_{1,p-1}=0_r$. All in, we have
\begin{equation*}
R_{2,1}=pI_r+Q_{1, p-1}.
\end{equation*}
Finally, as $a+q$ has order $p$ we have $R_{2,1}=0_r$ so that
\begin{equation*}
pI_r+Q_{1, p-1}=0. \tag{4.3.1}
\end{equation*}

Regarding $bb^c-b^cb$ as an $(r\times r)$ block matrix, we focus on the $(1,p-2)$ block. The matrices $a_0a_0^c$, $a_0q^c$, $a_0^ca_0$ and $a_0^cq$ make no contribution to our calculation as we calculate that they all have $0_r$ as their $(1,p-2)$ block. Thus we consider $qa_0^c-q^ca_0$. As the $(p-2)$th columns of the block matrices for $a_0$ and $a_0^c$ have a unique entry (namely the $(p-1,p-2)$ entry), we calculate that the $(1,p-2)$ block of $bb^c-b^cb$ is given by

\[Q_{1,p-1}(a_0^c)_{(p-1,p-2)}-(Q^c)_{(1,p-1)}{a_0}_{(p-1,p-2)}.\]
Applying $(4.3.1)$, using that $c$ is a block diagonal matrix and recognizing $C$ as an element of $\GL_r(\Z/p^2\Z)$ in the form described in $(4.1.1)$ yields that the $(1, p-2)$ block is in fact
\begin{align*}
(bb^c-b^cb)_{1,p-2}&=Q_{1,p-1}(a_0^c)_{(p-1,p-2)}-(Q^c)_{(1,p-1)}{a_0}_{(p-1,p-2)}\\
&=-pI_r(a_0^c)_{(p-1,p-2)}-(-C^{-(p-2)}pI_rC^{-(p-2)}){a_0}_{(p-1,p-2)}\\
&=-pI_r C^{p-2}(p-2)I_rC^{-(p-4)}+C^{-(p-2)}pI_rC^{-(p-2)}(p-2)I_r\\&= -p(p-2)C^2+p(p-2)C^{-2(p-2)}.\end{align*}
Since $b$ and $b^c$ commute, we have that $bb^c-b^cb=0_{(p-1)r}$. Therefore,
\[2pC^2-2pC^{-2(p-2)}=0_r.\]

Finally, since $p$ is an odd prime, we deduce that $pC^{2p-2}=pI_r\in \GL_r(\Z/p^2\Z)$. Hence, every entry off diagonal entry of $C^{2p-2}$ is divisible by $p$ while every diagonal entry is of the form $1+tp$. Therefore, $\bar{C}^{2p-2}$, which is projection of $C^{2p-2}$ to $\GL_r(p)$, is the identity matrix. But as $n$ is odd, $\bar{C}$ has order $p^r-1$ and we conclude that $r=1$, as desired.

Applying \cite[Proposition 3.8]{p.index3}, we have that $V_{p-2}(p)$ lifts all the way to a $\Q_p$-representation.
\end{proof}

\begin{proposition}\label{Steinberg}
Assume that $p$ is an odd prime and $G=\SL_2(p^r)$. Then $V_{p-1}(p^r)$ lifts to $\Z/p^2\Z$ if and only if $r=1$. Moreover, $V_{p-1}(p)$ lifts to $\Q_p$.
\end{proposition}
\begin{proof}
This time we set $n=p-1$ and we calculate in $\mathrm{M}_{pr}(\Z/p^2\Z)$. As observed above, the $R_{2,1}$ contribution comes from $pa_0$, $a_0qa_0^{p-2}$ and $qa_0^{p-1}$. Since $a_0$ is strictly lower triangular, we see that ${a_0}_{(2, i)}\ne 0_r$ only if $i=1$. Furthermore, again as $a_0$ is strictly lower triangular and as $n=p-1$, we see that ${(a_0^{p-2})}_{(i,j)}\ne 0_r$ only if $(i,j)\in\{(p-1,1), (p,1), (p,2)\}$, and $(a_0^{p-1})_{(i,j)}\ne 0_r$ only if $(i,j)=(p,1)$.

We deduce that \[R_{2,1}=p{a_0}_{(2,1)} + {a_0}_{(2,1)}.Q_{(1,p-1)}.(a_0^{p-2})_{(p-1, 1)}+{a_0}_{(2,1)}.Q_{1,p}.{(a_0^{p-2})}_{(p,1)}+Q_{2,p
}.(a_0^{p-1})_{(p,1)}.\]
We apply \cref{Pascal} to realize that 
\[{a_0}_{(2,1)} = I_r, \;\;\;(a_0^{p-2})_{(p-1, 1)}=(p-2)! I_r,\;\;\; (a_0^{p-2})_{(p,1)}=\binom{p-1}{p-3}(p-2)!\]
\[\text{ and } (a_0^{p-1})_{(p, 1)}=(p-1)!\]
Recall that $(p-1)!\equiv -1 \mod p$, $(p-2)!\equiv 1 \mod p$ and $\binom{p-1}{p-3}\equiv 1 \mod p$. Then
\begin{equation*}
R_{2,1}=pI_r+Q_{1, p-1}+Q_{1,p}-Q_{2,p}.
\end{equation*}
If $a+q$ has order $p$ then $R_{2,1}=0_r$ and we ascertain that
\begin{equation*}
pI_r+Q_{1, p-1}+Q_{1,p}-Q_{2,p}=0. \tag{4.4.1}
\end{equation*}

We focus on the $(1,p-1)$, $(1,p-2)$ and $(2,p-1)$ blocks of $bb^c-b^cb$, noticing that $a_0a_0^c$ and $a_0^ca_0$ have $0_r$ in these positions. Since $b$ and $b^c$ commute, we have that $bb^c-b^cb=0_{pr}$.

Observe that $a_0q^c$ and $a_0^cq$ have $0_r$ in the $(1, p-1)$ position. Hence, we consider the expression $qa_0^c-q^ca_0$ and investigate its $(1,p-1)$ block. Since $a_0$ and $a_0^c$ are strictly lower triangular, they have exactly one non-zero block in the $(p-1)$th column: the $(p, p-1)$ block. Moreover, $c$ is a block diagonal matrix. Using that $pQ_{i,j}=0_r$, we deduce that

\begin{eqnarray*}
(qa_0^c-q^ca_0)_{p, p-1}
&=&Q_{1, p}.(c^{-1})_{p,p}.{a_0}_{(p, p-1)}.c_{p-1, p-1}-(c^{-1})_{1,1}.Q_{1,p}.c_{p,p}.{a_0}_{(p, p-1)}\\
&=&Q_{1, p}.C^n. (p-1)I_r. C^{-(n-2)}- C^{-n}. Q_{1,p}. C^{-n}. (p-1)I_r\\
&=&-C^2Q_{1,p} + C^{-2(p-1)}Q_{1,p}\\
&=&0_r
\end{eqnarray*}

so that

\begin{equation*}
C^{-2(p-1)}Q_{1,p}=C^2.Q_{1,p}.\tag{4.4.2}
\end{equation*}

Performing similar calculations, examining the $(1,p-2)$ position gives

\begin{equation*}
(C^4-C^{-2(p-1)})Q_{1, p}+2(C^{2-2(p-1)}-C^2)Q_{1,p-1} = 0_r\tag{4.4.3}
\end{equation*}

and examining the $(2,p-1)$ position gives

\begin{equation*}
-(C^2+C^{2-2(p-1)})Q_{2, p}+ (C^{2-2(p-1)}-C^2)Q_{1,p-1} = 0_r.\tag{4.4.4}
\end{equation*}

Subtracting twice $(4.4.5)$ from $(4.4.4)$ yields

\begin{equation*}
(C^4-C^{-2(p-1)})Q_{1, p}+2(C^2+C^{2-2(p-1)})Q_{2, p} = 0_r.\tag{4.4.5} 
\end{equation*}

Multiplying $(4.4.1)$ by $2(C^2+C^{2-2(p-1)})(C^{2-2(p-1)}-C^2)$, eliminating $Q_{1,p-1}$ and $Q_{2,p}$ using $(4.4.3)$ and $(4.4.5)$ and finally simplifying using $(4.4.2)$, we obtain
\begin{equation*}
2p(C^{4-4(p-1)}-C^4)+2(C^8-C^6)Q_{1,p} = 0_r.\tag{4.4.6}
\end{equation*}

By $(4.4.2)$ we have that 
$(C^{-2(p-1)}-C^2)Q_{1,p}=0_r$ so that  $(4.4.6)$ yields

\begin{eqnarray*}
&& 2p(C^{4-4(p-1)}-C^4)(C^{-2(p-1)}-C^2)\\
&= & 2pC^6(C^{-4(p-1)}-I_r)(C^{-2p}-I_r)\\
&= & 0_r.
\end{eqnarray*}
Hence, as $p$ is odd, every entry of $C^6(C^{-4(p-1)}-I_r)(C^{-2p}-I_r)$ is a multiple of $p$. Recall that $\ov C$ is the image of $C$ under the natural projection from $\GL_{r}(\Z/p^2\Z)$ to $\GL_{r}(p)$. Then $\ov C$ is cyclic of order $p^r-1$. Since \[\ov C^6(\ov C^{-4(p-1)}-\ov I_r)(\ov C^{-2p}-\ov I_r)=\ov 0_r,\]
at least one of $\ov C^{-4(p-1)}-\ov I_r$ or $\ov C^{-2p}-\ov I_r$ is equal to $\ov 0_r$. Since $\ov C$ has order $p^r-1$, we deduce that $\ov C^{-4(p-1)}-\ov I_r=\ov 0_r$ and therefore $4(p-1)$ must be a multiple of  $p^r-1$. Hence, assuming that $r>1$,we deduce that $p=3$ and $r=2$. We verify computationally using MAGMA \cite{magma} that $V_2(9)$ does not lift when $p=3$ and $r=2$. In fact, we show that the derived subgroup of the subgroup of $\Aut(W)$ generated by the elements of the cosets $aQ$ and $bQ$, where $H=\langle a,b\rangle$, has a quotient isomorphic to a non-split extension $3^6.\PSL_2(9)$. We obtain a contradiction and so $r=1$. This proves the first statement of the result.

We observe that $V_{p-1}(p)$ is irreducible of dimension $p$ and so is the Steinberg module for $\SL_2(p)$. In particular, $V_{p-1}(p)$ is projective, and by \cref{Projective} we deduce that $V_{p-1}(p)$ lifts to a $\Q_p\SL_2(p)$-representation.
\end{proof}

We now turn our attention to the module $V_p(p^r)$. The structure of the following proof is analogous to \cref{Steinberg} and so for the sake of exposition, we conceal some of the details.

\begin{proposition}\label{noplift}
Assume that $p$ is a prime and $G=\SL_2(p^r)$. Then $V_p(p^r)$ lifts to $\Z/p^2\Z$ if and only if $p^r=2$. In particular, $V_2(2)$ lifts to $\Q_2$.
\end{proposition}
\begin{proof}
Assume that $p^r=2$. Then $V_2(2)\cong V_1(2)\oplus V_0(2)$, where $V_0(2)$ is the trivial module for $\FF_2G$. By \cref{Steinberg}, we have that $Y\oplus \Q_2$ is a lift of $V_2(2)$ to $\Q_2$ where $Y$ is the lift of $V_1(2)$ to $\Q_2$ and $\Q_2$ is a $1$-dimensional trivial module for $\Q_2 G$.

For the remainder of the proof, aiming for a contradiction, we suppose that $p^r>2$. This time, we obtain the equation
\begin{equation*}
0_r=R_{2,1}=pI_r+Q_{1, p-1}+Q_{1,p}+Q_{1, p+1}(a_0^{p-2})_{(p+1,1)}-Q_{2,p}. \tag{4.5.1}
\end{equation*}

The exact value of $(a_0^{p-2})_{(p+1,1)}$ can be given, but it is not important for our purposes.

We focus on the $(1,p-1)$, $(2,p-1)$, $(2,p)$, $(2, p+1)$ and $(3, p)$ blocks of $bb^c-b^cb$, noticing that $a_0a_0^c$ and $a_0^ca_0$ have $0_r$ in these positions. Since $b$ and $b^c$ commute, we have that $bb^c-b^cb=0_{pr}$.

An analysis of the $(1, p-1)$ position gives

\begin{equation*}
C^{-2p}Q_{1,p}- C^2Q_{1, p+1} = 0_r.
\end{equation*}

Examining the $(2,p)$ position gives

\begin{equation*}
(C^{-2p}-C^2)Q_{1,p} = 0_r.
\end{equation*}

The $(2, p-1)$ position yields

\begin{equation*}
(C^{4-2p}-C^2)(Q_{2, p}+Q_{1, p-1}) = 0_r.
\end{equation*}

and examining the $(2,p+1)$ position gives

\begin{equation*}
(C^{-2p}-C^2)Q_{1,p+1} = 0_r.
\end{equation*}

Finally, the $(3,p)$ position produces

\begin{equation*}
(C^{2-2p}-C^4)Q_{1, p}+2(C^{4-2p}-C^2)Q_{2, p} = 0_r.
\end{equation*}

Noticing the special roles the expressions $2(C^{4-2p}-C^2)$ and $C^{-2p}-C^2$ play in annihilating the matrices $Q_{i,j}$, applying the above relations to $(4.5.1)$ we get for every odd prime $p$:

\begin{equation*}
2p(C^{4-4p}-C^{2-2p}-C^{6-2p}-C^4) = 0_r.\tag{4.5.2}
\end{equation*}

When $p=2$ we get the simpler equation

\begin{equation*}
2(C^{-4}-C^2)(1-C^2)= 0_r.
\end{equation*}

When $p=2$, we deduce that $\bar{C}^6 =I_r\in \GL_r(2)$, where $\bar{C}$ is the image of $C$ under the projection to $\mathrm{M}_{r}(\FF_2)$. Since $\bar{C}$ has order $2^r-1$ and $r>1$, we deduce that $r=2$. We treat this case computationally using MAGMA \cite{magma}. Explicitly, we show that the subgroup of $\Aut(W)$ generated by the elements of the cosets $aQ$ and $bQ$, where $H=\langle a,b\rangle$, has a quotient isomorphic to $\SL_2(5)$ and so there is no appropriate $\SL_2(4)$ subgroup of $\Aut(W)$. Hence, there is no lift in this case.

When $p$ is odd, observe that $(4.5.2)$ holds for any power of $c$. Indeed, we have that $c^{p^{r-1}+\dots+1}$ has order $p-1$ and we deduce that
\[2p(1-1-C^{4(p^{r-1}+\dots+1)}-C^{4(p^{r-1}+\dots+1)}) = -4p(C^{4(p^{r-1}+\dots+1)})= 0_r\] when $p$ is odd. But then $pC^{4(p^{r-1}+\dots+1)}$ is the zero matrix in $\mathrm{M}_{r}(\Z/p^2\Z)$, so that all entries of $C^{4(p^{r-1}+\dots+1)}$ are divisible by $p$. Then, for $\bar{C}$ the image of $C$ under the projection to $\mathrm{M}_{r}(\FF_p)$, we have that $\bar{C}^{4(p^{r-1}+\dots+1)}$ is the zero matrix in $\mathrm{M}_{r}(\FF_p)$, a contradiction as $\bar{C}$ is invertible.
\end{proof}

We complete the proof of \cref{MT} by determining when $\Lambda(p^r)$ lifts.

\begin{proposition}
Assume that $p$ is a prime and $G=\SL_2(p^r)$. Then $\Lambda(p^r)$ lifts to $\Z/p^2\Z$ if and only if $p^r=2$. In particular, $\Lambda(2)$ lifts to $\Q_2$.
\end{proposition}
\begin{proof}
We observe that when $p^r=2$, $\Lambda(2)\cong V_2(2)\cong V_1(2)\oplus V_0(2)$ where $V_0(2)$ is the trivial module for $\FF_2\SL_2(2)$. Then the result for $p^r=2$ follows directly from \cref{noplift}. Hence, we may suppose that $p^r>2$ for the remainder of the proof.

Aiming for a contradiction, assume that $\Lambda(p^r)$ lifts to $\Z/p^2\Z$. We define a dual $G$-action on $(\Z/p^2\Z)^{p+1}$ by sending $g\in G\le \GL_{(p+1)r}(\Z/p^2\Z)$ to $(g^{-1})^T$.  Write $\hat{W}$ for the $(\Z/p^2\Z)G$-module resulting from this new action. Then $g$ acts on $\hat{W}/p\hat{W}$ as $(g^{-1})^T$ acts on $\Lambda(p^r)$ from which we deduce that $\hat{W}/p\hat{W}\cong V_p(p^r)$ as an $\FF_p G$-module. But then $\hat{W}$ would be a lift of $V_p(p^r)$, which contradicts \cref{noplift}.
\end{proof}

\bibliographystyle{alpha}
\bibliography{my.books}
\end{document}